\documentclass[]{amsart}

\usepackage[all]{xy}
\usepackage{amscd,amsthm,amssymb,amsfonts,amsmath,euscript}
\usepackage{comment}

\theoremstyle{plain}
\newtheorem{thm}{Theorem}
\newtheorem{lemma}[thm]{Lemma}
\newtheorem{prop}[thm]{Proposition}
\newtheorem{cor}[thm]{Corollary}

\theoremstyle{definition}
\newtheorem{defn}[thm]{Definition}
\newtheorem{eg}[thm]{Example}

\theoremstyle{remark}
\newtheorem{remark}[thm]{Remark}

\newcommand{\nc}{\newcommand}

\numberwithin{equation}{section}

\def\makeop#1{\expandafter\def\csname#1\endcsname
  {\mathop{\rm #1}\nolimits}\ignorespaces}
\makeop{Hom}   \makeop{End}   \makeop{Aut}   \makeop{Isom}  \makeop{Pic} 
\makeop{Gal}   \makeop{ord}   \makeop{Char}  \makeop{Div}   \makeop{Lie} 
\makeop{PGL}   \makeop{Corr}  \makeop{PSL}   \makeop{sgn}   \makeop{Spf}
\makeop{Spec}  \makeop{Tr}    \makeop{Nr}    \makeop{Fr}    \makeop{disc}
\makeop{Proj}  \makeop{supp}  \makeop{ker}   \makeop{im}    \makeop{dom}
\makeop{coker} \makeop{Stab}  \makeop{SO}    \makeop{SL}    \makeop{SL}
\makeop{Cl}    \makeop{cond}  \makeop{Br}    \makeop{inv}   \makeop{rank}
\makeop{id}    \makeop{Fil}   \makeop{Frac}  \makeop{GL}    \makeop{SU}
\makeop{Nrd}   \makeop{Sp}    \makeop{Tr}    \makeop{Trd}   \makeop{diag}
\makeop{Res}   \makeop{ind}   \makeop{depth} \makeop{Tr}    \makeop{st}
\makeop{Ad}    \makeop{Int}   \makeop{tr}    \makeop{Sym}   \makeop{can}
\makeop{length}\makeop{SO}    \makeop{torsion} \makeop{GSp} \makeop{Ker}
\makeop{Adm}   \makeop{Mat}
\def\makebb#1{\expandafter\def
  \csname bb#1\endcsname{{\mathbb{#1}}}\ignorespaces}
\def\makebf#1{\expandafter\def\csname bf#1\endcsname{{\bf
      #1}}\ignorespaces} 
\def\makegr#1{\expandafter\def
  \csname gr#1\endcsname{{\mathfrak{#1}}}\ignorespaces}
\def\makescr#1{\expandafter\def
  \csname scr#1\endcsname{{\EuScript{#1}}}\ignorespaces}
\def\makecal#1{\expandafter\def\csname cal#1\endcsname{{\mathcal
      #1}}\ignorespaces} 

\def\doLetters#1{#1A #1B #1C #1D #1E #1F #1G #1H #1I #1J #1K #1L #1M
                 #1N #1O #1P #1Q #1R #1S #1T #1U #1V #1W #1X #1Y #1Z}
\def\doletters#1{#1a #1b #1c #1d #1e #1f #1g #1h #1i #1j #1k #1l #1m
                 #1n #1o #1p #1q #1r #1s #1t #1u #1v #1w #1x #1y #1z}
\doLetters\makebb   \doLetters\makecal  \doLetters\makebf
\doLetters\makescr 
\doletters\makebf   \doLetters\makegr   \doletters\makegr

     \def\qed{\qedmark\medbreak}%
\def\qedmark{{\enspace\vrule height 6pt width 5pt depth 1.5pt}}%
    \def\setminus{\smallsetminus}

\normalsize

\makeop{Bl}

\def\Spec{{\rm Spec}\,}

\newcommand{\Z}{\mathbb Z}
\newcommand{\Q}{\mathbb Q}
\newcommand{\R}{\mathbb R}
\newcommand{\C}{\mathbb C}

\newcommand{\isoto}{\stackrel{\sim}{\longrightarrow}}
\nc{\embed}{\hookrightarrow}

\newcommand{\ch}{characteristic }

\nc{\ol}{\overline}
\nc{\wt}{\widetilde}
\nc{\opp}{\mathrm{opp}}

\makeop{Ram}
\makeop{Rep}

\begin{document}
\renewcommand{\thefootnote}{\fnsymbol{footnote}}
\setcounter{footnote}{-1}



\title[Japanese Dedekind domains]{Japanese Dedekind domains are excellent} 
\author{Chia-Fu Yu}
\address{
Institute of Mathematics, Academia Sinica \\
Astronomy Mathematics Building \\
No.~1, Roosevelt Rd. Sec.~4 \\ 
Taipei, Taiwan, 10617} 
\email{chiafu@math.sinica.edu.tw}



\date{\today}


\begin{abstract}
  The well-known fundamental identity in number theory
  expresses the degree of an extension of
  global fields in terms of local information. In this article 
  we show a generalized 
  fundamental identity for arbitrary Dedekind domains.
  As an application, we show that any Japanese Dedekind
  domain is excellent. 
\end{abstract} 

\maketitle


\section{Introduction}
\label{sec:introduction}

Let $A$ be an $S$-ring of integers of a global field $K$, where $S$ is
a nonempty finite set of places of $K$ containing all Archimedean
ones. 
The well-known fundamental identity in number theory states that for
any finite field extension $L/K$ and any nonzero prime ideal $\grp$ of
$A$, one has 
\begin{equation}
  \label{eq:fe}
  \sum_{i=1}^r e_i f_i = [L:K],
\end{equation}
where $e_i$ and $f_i$ are the ramification index and residue class
degree of the prime ideals $\grP_i$ lying over 
$\grp$ (for $1\le i \le r$), respectively. 
If $A$ is an arbitrary Dedekind domain, then the fundamental identity
is no longer true but instead one has the fundamental inequality
$\sum_i e_i f_i\le [L:K]$ in general. It is well known that when the
integral closure $B$ of $A$ is $L$ is a finitely generated $A$-module,
the equality holds; on the other hand the strict inequality can occur;
see \cite[Remark p.~15]{serre:lf}.  

One can rephrase the statement in terms of valuation theory,
cf.~\cite[Chap.~II, Proposition 8.5, p.~165]{neukirch:ant}.
In this reformation, Cohen and
Zariski~\cite{cohen-zariski} proved 
a fundamental inequality for extensions of
an arbitrary valuation $v$ of a field $K$ to a finite field extension 
$L/K$. Earlier Schmidt \cite{schmidt:1936}
constructed a valuation ring $A$ in $K$ which is non-Japanese,
that is, the integral closure $B$ of $A$ in some finite extension
$L/K$ is not finite over $A$.
Furthermore, 
Schmidt's example realizes the strict
inequality in the fundamental inequality, 
cf.~\cite{cohen-zariski}. 

In this paper we show a modified fundamental identity
for arbitrary Dedekind domains as follows.


\begin{thm}\label{thm:1}
  Let $A$ be a Dedekind domain with quotient field $K$, $L$ a finite
  extension field of $K$ with integral closure $B$ of $A$. For any
  nonzero prime ideal $\grp$ of $A$ one has the equality 
  \begin{equation}
    \label{eq:1.1}
  \sum_{\grP|\grp} e_\grP f_\grP = [(L\otimes_K K_\grp^*)_{\rm ss}:
  K_\grp^*],    
  \end{equation}
where $K_\grp^*$ denotes the completion of $K$ at $\grp$ and
$(L\otimes_K K_\grp^*)_{\rm ss}$ denotes the semi-simplification
of $L\otimes_K K_\grp^*$.  
\end{thm}

We apply Theorem~\ref{thm:1} and prove the following new result.


\begin{thm}\label{thm:2}
  Let $A$ be a Dedekind domain with quotient field $K$. Then the
  following statements are equivalent. 
  \begin{enumerate}
  \item $A$ is excellent.
  \item $A$ is quasi-excellent.
  \item $A$ is a Nagata ring.
  \item $A$ is universally Japanese.
  \item $A$ is a Japanese ring.
  \end{enumerate}
\end{thm}

We shall recall (quasi-)excellent rings, Nagata rings and
(universally) Japanese rings, and some properties 
in Section~\ref{sec:02} and refer to \cite{matsumura:ca80} for more 
details.
In particular, one has the following well-known relations:
\[
  \begin{split}
    \text{(excellent rings)} & \implies  \text{(quasi-excellent rings)}
       \implies  \text{(Nagata rings)} \iff \\ 
       & \text{(Noetherian universally Japanese rings)}\implies
       \text{(Japanese rings)}.     
  \end{split}
  \] 
By Theorem~\ref{thm:2}, when $A$ is a Dedekind domain,  
the above classes of rings are actually the same. In particular, 
excellent Dedekind domains are the same as Japanese Dedekind domains,
while the definition of the latter is much simpler. We remark that 
excellent Dedekind domains 
play a subtle role in the
construction of moduli spaces: particularly in Artin's
approximation theorem \cite{artin:approx1} and   
the existence of N\'eron models \cite[10.2, Theorem 2,
p.~297]{BLR:neron}. 
Thus, for some arithmetic applications
Theorem~\ref{thm:2} provides a simpler way to access  
excellent Dedekind domains.  



For a property $\bbP$ of commutative rings, we say $A$ is locally
$\bbP$ if the localization $A_\grp$ at $\grp$ for every $\grp\in \Spec A$
satisfies the property $\bbP$. Applying Theorem~\ref{thm:2} to the
case where $A$ 
is a discrete valuation ring, we obtain the following

\begin{cor}\label{cor:3}
  Let $A$ be a Dedekind domain with quotient field $K$. Then the
  following statements are equivalent. 
  \begin{enumerate}
  \item $A$ is locally excellent.
  \item $A$ is locally quasi-excellent.
  \item $A$ is a locally Nagata ring.
  \item $A$ is locally universally Japanese.
  \item $A$ is a locally Japanese ring.
  \end{enumerate}
\end{cor}

R. Heitmann shows that there is a locally Nagata principal ideal
domain that is not Nagata. Therefore, the equivalent classes of rings
in Theorem~\ref{thm:2} are strictly stronger than the equivalence
classes of rings in Corollary~\ref{cor:3}. Note that 
when the equality in \eqref{eq:fe} holds, Theorem~\ref{thm:1} 
gives a direct connection between $G$-rings and
locally Japanese rings; see Proposition~\ref{prop:fe_g}. 
This is how Theorem~\ref{thm:2} is proved.

We explain where \eqref{eq:1.1} comes from. 
Let us fix a nonzero prime ideal $\grp$ of $A$. On $B$ we have two linear
topologies induced by two systems of the subgroups $\{\grp^n \Lambda \}_n$
and $\{\grp^n B\}_n$, respectively, 
where $\Lambda$ is a free finite $A$-submodule of full
rank in $B$, and they agree if the localization $B_\grp$ at $\grp$ 
is finite over $A_\grp$ (and as we will
see from Theorem~\ref{thm:1} that this is actually ``if and only if''). 
The completion of $B$ for the first topology gives 
$B\otimes_A A^*_\grp$, 
which has rank $[L\otimes_K K_\grp^*:K_\grp^*]=[L:K]$ over $A_\grp^*$. 
The completion for the second topology gives
$B_\grp^*=\prod_{\grP|\grp} B_{\grP}^*$, which has rank
$\sum_{\grP|\grp} e_\grP f_\grP$, cf.~the proof of Lemma~\ref{lm:fe}. 
It is clear that there is a
surjective map from $B\otimes_A A^*_\grp$ to $B_\grp^*$,  and using
the valuation theory (see \cite{bourbaki:ca}, also
see~Section~\ref{sec:03})
we show that $B_\grp^*$
is exactly the reduced ring of $B\otimes_A A^*_\grp$. 
The rank of the latter one is
equal to $[(L\otimes_K K_\grp^*)_{\rm ss}: K_\grp^*]$.


  

This paper is organized as follows. In Section~\ref{sec:02} we recall several
rings mentioned above and their relations. 
The proofs of
Theorems~\ref{thm:1} and~\ref{thm:2} are given in Section~\ref{sec:04}.

\section{Japanese, Nagata, (quasi-)excellent and $G$-rings}
\label{sec:02}

In this section we recall the definition of several fundamental rings
in Introduction and their relations.
Our references are Matsumura 
\cite{matsumura:ca80}, and EGA IV \cite{ega4:20,ega4:24}, also
cf.~\cite[Section 2]{yu:mo}. 
All rings and algebras in this section are commutative with identity.


\subsection{Japanese, universally Japanese and Nagata rings}
\label{sec:21}

\begin{defn}\label{def:N}
  Let $A$ be an integral domain with quotient field $K$.
\begin{enumerate}
  \item  We say that {\it $A$ is N-1} if the integral closure $A'$ of
    $A$ in 
    its quotient field $K$ is a finite $A$-module. 
  \item We say that {\it $A$ is N-2} if for any finite field extension
    $L$ over $K$, the integral closure $A_L$ of $A$ in $L$ is a finite
    $A$-module.   
\end{enumerate}
\end{defn}

The first non-N-1 one-dimensional integral domain was constructed by
Akizuki \cite{akizuki}, cf.~\cite{reid:uca}. K.~Schmidt
\cite{schmidt:1936} and respectively Nagata (Appendix: Examples of bad
Noetherian rings of \cite{nagata:local}) constructed different
Dedekind domains which are not N-2.  

\def\Max{{\rm Max}}

\begin{defn}\label{def:nagata}
  A ring $A$ is said to be {\it Nagata} if 
\begin{enumerate}
  \item $A$ is Noetherian, and  
  \item $A/\grp$ is N-2 for any prime ideal $\grp$ of $A$. 
\end{enumerate}
\end{defn}

Nagata rings are the same
as what are called Noetherian universally Japanese rings in EGA IV
\cite[23.1.1, p.~213]{ega4:20}.  

\begin{defn}\
\begin{enumerate}
\item An integral domain $A$ is said to be {\it Japanese} if it is N-2.
\item A ring $A$ is said to be {\it Japanese}
if for any minimal prime ideal $\grp$ of $A$, the quotient domain
$A/\grp$ is N-2.    
\item A ring $A$ is said to be {\it universally Japanese} if any finitely
  generated integral domain over $A$ is Japanese.  
\end{enumerate}
\end{defn}

We clarify two notions of local finiteness of modules as follows.
\begin{defn}
  Let $M$ be a module over a commutative ring $A$.

{(1)} We say that $M$ is \emph{Zariski-locally finitely generated}
if for any prime $\grp\in \Spec A$, there is an element $f\in A$ such
that $f\not\in \grp$ and $M_f$ is a finite $A_f$-module.

{(2)} We say that $M$ is \emph{locally finitely generated}
if for any prime $\grp\in \Spec A$, the localization $M_\grp$ at
$\grp$ is a finite $A_\grp$-module.
\end{defn}

One easily shows that if $M$ is Zariski-locally finitely generated,
then $M$ is finitely generated. Indeed, for each maximal ideal $\grm$
of $A$ there is an element $f\in A$ such that $\grm \not\in f$ and
$M_f$ is a finite $A_f$-module. Since $\Spec A$ is quasi-compact,
there exist elements $f_1, \dots, f_n \in A$ such that $(f_1,\dots,
f_n)A=A$ and each $M_{f_i}$ is a finite $A_{f_i}$-module. Let
$S_i\subset M$ is a finite set of generators of  $M_{f_i}$ over 
$A_{f_i}$. Then the union of all $S_i$ is a finite set of generators
of $M$ over $A$.      

\begin{lemma}\label{Z-loc-Japan}
  If for any maximal ideal $\grm \in \Max (A)$, there exists an
  element $f\in A$ such that $f\not\in \grm $ and $A_f$ is Japanese,
  then $A$ is Japanese.   
\end{lemma}
\begin{proof}
  Replacing $A$ by $A/\grp$ for each minimal prime $\grp$, we may
  assume that $A$ is an integral domain. Let $L/K$ be a finite field
  extension with integral closure $B$ of $A$ in $L$. 
  Since the construction of normalization commutes with localization, 
  our assumption
  implies that for any $\grm\in \Max(A)$, there exists an element
  $f\in A$ such that $f \not\in \grm$ and that $B_f$ is a finite
  $A_f$-module. That is, $B$ is Zariski-locally finitely generated and
  hence $B$ is finitely generated. \qed  
\end{proof}

\begin{remark} 
  The proof of Lemma~\ref{Z-loc-Japan}
  does not show that any locally Japanese ring (that is, its
  localization at every prime is Japanese) is Japanese. 
  The following is an example, 
  due to Nagata, 
  of a locally finitely generated module which
  is not finitely generated. 
  Let $k$ be a field of
  characteristic $p>0$ such that $[k:k^p]=\infty$. Let
  $R:=k^p[[x,y]][k]$, where $x$ and $y$ are indeterminates. 
  Then the completion $R^*$ of $R$ at the maximal
  ideal $(x,y)$ is equal to $k[[x,y]]$ and $R\subsetneq R^*$. Let
  $b_1, \dots, b_n, \dots$ be a sequence of $p$-independent elements
  in $k$, and $p_1,\dots, p_m, \dots$ be 
  mutually non-associative prime elements (i.e.~$(p_i)\neq (p_j)$
  for $i\neq j$). Put $q_n:=p_1\cdots p_n$. Let $c:=\sum_{i=1}^\infty
  b_i q_i \in R^*$ and let $T:=R[1/x][c]$ whose normalization is denoted
  by $T'$. Then for every prime ideal $\grp$ of $T$, the localization 
  $T'_\grp$ is a finite $T_\grp$-module while $T'$ is not finite over
  $T$; see \cite[A1.Example 8, p.~211]{nagata:local} for more details.   
\end{remark}

\subsection{G-rings, closedness of singular loci and excellent rings}
\label{sec:32}

\begin{defn}(\cite[\S\,33, p.~249]{matsumura:ca80}). 
\begin{enumerate}
  \item Let $A$ be a Noetherian ring containing a field $k$. We say
    that $A$ is {\it geometrically regular over $k$} if for any finite
    field extension $k'$ over $k$, the ring $A\otimes_k k'$ is 
    regular \cite[p.~78]{matsumura:ca80}. 
This is equivalent to say that the local ring $A_\grm$
    has the same property for all maximal ideals $\grm$ of $A$.
  \item Let $\phi:A\to B$ be a homomorphism (not necessarily of finite
    type) of Noetherian rings. We say that $\phi$ is {\it regular} if
    it is 
    flat and for
    each $\grp\in \Spec A$, the fiber ring $B\otimes_A k(\grp)$ is
    geometrically regular over the residue field $k(\grp)$. 
  \item A Noetherian ring $A$ is said to be a {\it G-ring} if for each $\grp
    \in \Spec A$, the natural map 
    $\phi_\grp: A_\grp\to (A_\grp)^*$ is regular,
    where $(A_\grp)^*$ denotes the completion of the local ring
    $A_\grp$. 
\end{enumerate}
\end{defn}

Note that the natural map $\phi_\grp: A_\grp\to (A_\grp)^*$ is faithfully
flat. The fibers of the natural morphism $\Spec (A_\grp)^*\to \Spec
A_\grp$ are called \emph{formal fibers}. To say a Noetherian ring $A$ is a
$G$-ring then is equivalent to saying that all formal fibers of the
canonical map $\phi_\grp$ for each prime ideal $\grp$ of $A$ 
are geometrically regular. 
It is clear that, if $A$ is a G-ring, then any localization $S^{-1}A$
of $A$ and any homomorphism image $A/I$ of $A$ are G-rings.


\begin{lemma}\label{lm:geom-reg}
  Let $K/k$ be any field extension. Then $K$ is geometrically regular
  over $k$ if and only if $K/k$ is separable.    
\end{lemma}

The field extension $K/k$ is separable if and only if
for any finite field extension $k'/k$, the tensor product $k'\otimes_k
K$ is reduced or equivalently that $k'\otimes_k K$ is regular. This
proves the lemma.  

\def\reg{{\rm Reg}}

For a Noetherian scheme $X$, let $\reg (X)$ denote the subset of $X$
that consists of regular points, which is called the regular locus of $X$.
 
\begin{defn} Let $A$ be a Noetherian ring. 
  \begin{enumerate}
  \item We say that $A$ is {\it J-0} if $\reg (\Spec A)$ contains a
    nonempty open set of $\Spec A$.
  \item We say that $A$ is {\it J-1} if $\reg (\Spec A)$ is open in  
    $\Spec A$.
  \end{enumerate}
\end{defn}

\begin{thm}\label{j2} \cite[Theorem 73, p.~246]{matsumura:ca80}.
  For a Noetherian ring $A$, the following conditions are equivalent:
  \begin{enumerate}
  \item[(a)] any finitely generated $A$-algebra $B$ is J-1;
  \item[(b)] any finite $A$-algebra $B$ is J-1;
  \item[(c)] for any $\grp\in \Spec A$, and for any finite radical
    extension $K'$ of $k(\grp)$, there exists a finite $A$-algebra
    $A'$ satisfying $A/\grp\subseteq A'\subseteq K'$ which is J-0 and
    whose quotient field is $K'$. 
  \end{enumerate}
\end{thm}

\begin{defn} A Noetherian ring $A$ is said to be {\it J-2} if it
  satisfies one 
  of the equivalent conditions in Theorem~\ref{j2}.
\end{defn}

\begin{lemma}\label{lm:japan-j2}
  Any Noetherian Japanese ring $A$ of dimension one is J-2.  
\end{lemma}
\begin{proof}
  For each $\grp\in \Spec A$, the quotient domain
  $A/\grp$ is either a field or a
  Noetherian Japanese domain of dimension one. In the first case,
  the condition (c) holds trivially by taking $A'=K'$.
  In the second case, the integral closure $A'$ of $A$ in
  $K'$ is finite over $A$ and is a Dedekind domain, which particularly
  is J-0. 
  Therefore, $A$ is J-2. \qed
\end{proof}



\begin{thm}\label{g1}\
  \begin{enumerate}
  \item Any complete Noetherian local ring is a G-ring. 
  \item If for any maximal ideal $\grm$ of a Noetherian ring $A$, the
    natural map $A_\grm\to (A_\grm)^*$ is regular, then $A$ is a
    G-ring
  \item Let $A$ and $B$ be Noetherian rings, and let $\phi:A\to B$ be
    a faithfully flat and regular homomorphism. If $B$ is J-1, then
    so is $A$.
  \item Any semi-local G-ring is J-1.    
  \end{enumerate}
\end{thm}
\begin{proof}
  (1) See \cite[Theorem 68, p.~225 and p.~250]{matsumura:ca80}. (2)
      See \cite[Theorem 75, p.~251]{matsumura:ca80}. (3) and (4) See 
      \cite[Theorem 76, p.~252]{matsumura:ca80}. 
\end{proof}

\begin{thm}\label{g2}\
  \begin{enumerate}
  \item Let $A$ be a G-ring and $B$ a finitely generated
    $A$-algebra. Then $B$ is a G-ring.
  \item Let $A$ be a G-ring which is J-2. Then $A$ is a Nagata ring.
  \end{enumerate}
\end{thm}
\begin{proof}
  (1) See \cite[Theorem 77, p.~254]{matsumura:ca80}. 
  (2) See \cite[Theorem 78, p.~257]{matsumura:ca80}.
\end{proof}

\begin{defn}\label{def:excellent} \cite[\S\, 34, p.~259]{matsumura:ca80}.
Let $A$ be a Noetherian ring.
  \begin{enumerate}
  \item We say that $A$ is {\it quasi-excellent} if the following
    conditions are satisfied:
    \begin{itemize}
    \item [(i)] $A$ is a G-ring;
    \item [(ii)] $A$ is J-2.
    \end{itemize}
  \item We say that $A$ is {\it excellent} if it satisfies (i), (ii)
    and the following condition
    \begin{itemize}
    \item [(iii)] $A$ is universally catenary 
\cite[p.~84]{matsumura:ca80}. 
    \end{itemize}
  \end{enumerate}
\end{defn}

\begin{remark}\label{rem:catenary}
If $A$ is catenary, then so are any localization of $A$ and any 
homomorphism image of $A$. To show a Noetherian ring $A$ is 
universally catenary, it then suffices to show that every polynomial
ring  $A[X_1,\dots, X_n]$, for $n\ge 1$, is catenary.
If $A$ is Cohen-Macaulay, then $A$ is catenary and $A[X]$ is again
Cohen-Macaulay cf.~\cite[Proposition 18.9 and Corollary
18.10]{eisenbud}.  
Therefore, any Cohen-Macaulay ring is universally
catenary. Since every regular ring is Cohen-Macaulay, every regular
ring is universally catenary.   
\end{remark}

\begin{remark}\
\begin{enumerate}
\item Each of the conditions (i), (ii), and (iii) is stable under the
localization and passage to a finitely generated algebra
(Theorems \ref{j2} (1) and \ref{g2} (1)).
\item Note that (i), (ii), (iii) are conditions depending only on
  $A/\grp$, for $\grp\in
  \Spec A$. Thus a Noetherian ring $A$ is (quasi-)excellent if and
  only if so is $A_{\rm red}$.
\item The conditions (i) and (iii) are of local nature (in the sense
  that if they hold for $A_\grp$ for all $\grp\in \Spec A$, then they
  hold for $A$), while the condition (ii) is not. 
\item  Theorem~\ref{g2} (2) states that any quasi-excellent ring is a
  Nagata ring. 

\item It follows from Theorems~\ref{j2} and \ref{g1} (4) that any
  Noetherian local G-ring is quasi-excellent.


\item Nagata's example of a $2$-dimensional Noetherian local
  ring that is catenary but not universally catenary
  \cite[(14.E), p.~87]{matsumura:ca80} is a $G$-ring,
  and is also a J-2 ring as any local $G$-ring is a J-2 ring. So it is
  a quasi-excellent catenary local ring that is not excellent.

\item Rotthaus \cite{rotthaus} 
  constructed a regular local ring $R$ of dimension three
  which contains a field and which is Nagata, but not quasi-excellent.

\end{enumerate}
\end{remark}

\section{Valuations, completions and extensions}
\label{sec:03}

In this section, a group will
mean an abelian group unless stated otherwise.

\subsection{Valuations and valuation rings}
\label{sec:V.1}

For a totally ordered abelian group $\Gamma$ written additively, we
denote by $\Gamma_\infty=\Gamma\cup \{\infty\}$ the totally ordered
commutative monoid with 
$\gamma \le \infty$ for all $\gamma \in \Gamma$ and
$\gamma+\infty=\infty+\gamma=\infty$ for all $\gamma\in
\Gamma_\infty$.   

\begin{defn}\label{def:v}
  \ 

  \begin{enumerate}
  \item  An integral domain $A$ with quotient field $K$ is called a
      \emph{valuation ring} or a \emph{valuation ring of $K$} if for
      any $x\in K^\times$ either $x\in A$ or $x^{-1}\in A$.

  \item  A {\it valuation} of a field $K$ is a group homomorphism
      $v:K^\times \to \Gamma$, where $\Gamma$ is a totally ordered
      abelian group, such that 
  \begin{equation}
    \label{eq:V.1}
v(x+y)\ge \min\{v(x), v(y)\}, \quad \forall\, x,y \in K^\times.       
  \end{equation}
  We extend $v$ to a function $v:K\to \Gamma_\infty $ by putting
  $v(0)=\infty$. Clearly, the condition \eqref{eq:V.1} holds for all
  $x,y\in K$. The homomorphism image $v(K^\times)$ is called the {\it
  value group} of the valuation $v$. Clearly, 
  $A:=\{x\in K : v(x)\ge 0\}$ is a valuation ring and $\grm:=\{x\in A:
  v(x)>0\}$ is its maximal ideal. We call $A$ and $\kappa:=A/\grm$ the
  \emph{valuation ring} and \emph{residue field} of $v$,
  respectively.

\item  A valuation $v$ of $K$ is said to be \emph{discrete} if its
      value group is isomorphic to $\Z$ compatible with the orders.

\item  Two valuations $v_1$ and $v_2$ of $K$ with value groups
      $\Gamma_1$ and $\Gamma_2$ are said to be \emph{equivalent} if
      there is an isomorphism $\alpha:\Gamma_1 \isoto \Gamma_2$ of
      ordered groups such that $v_2= \alpha\circ v_1$.     
  \end{enumerate}
\end{defn}

The construction $v\mapsto A$ gives rise to a map from the set of
equivalence classes of valuations of $K$ to the set of valuation rings
of $K$. 
The reverse construction is as follows: For a given valuation ring
$(A,\grm)$, define $\Gamma:=K^\times/A^\times$ and $P:=\grm/A^\times$
the set of positive elements, then $\Gamma$ is a totally ordered group
and the natural projection $v:K^\times\to \Gamma$ is a valuation of
$K$ so that the valuation ring of $v$ is equal to $A$. It is easy to
see the above map is bijection, cf. \cite[VI, \S 3.2,
Proposition~3]{bourbaki:ca}. 

\begin{prop}\label{prop:v-prime}
  Let $A$ be a valuation ring of a field $K$.

  \begin{enumerate}
  \item  The set of primes ideals of $A$ is totally ordered by the order
      of inclusion.

  \item  If $B\supset A$ is a subring of $K$, then $B$ is a valuation
      ring and the maximal ideal
      $\grm(B)$ of $B$ is a prime ideal of $A$. Moreover, the map
      $B\mapsto \grm(B)$ is a order-reversing bijection between the
      totally ordered set of subrings of $K$ containing $A$ and the
      totally ordered set of prime ideals of $A$. The inverse map is
      given by $\grp\mapsto A_\grp$, the localization of $A$ at
      $\grp$.   
  \end{enumerate}
\end{prop}

\begin{proof}
  (1) See \cite[VI, \S 1.2, Theorem~1(e)]{bourbaki:ca}. (2) See
  \cite[VI, \S 
  4.1, Proposition~1 and Corollary]{bourbaki:ca}. \qed
\end{proof}

\begin{defn}\label{isolate}
  A subgroup $H$ of an ordered group $G$ is said to be \emph{isolated}
  if the relation $0\le y\le x$ with $x\in H$ implies 
  $y\in H$.

\end{defn}

\begin{prop}\label{prop:quot}
  Let $G$ be an ordered group and $P$ the set of its positive elements.

  \begin{enumerate}
  \item 
   The kernel of an increasing homomorphism of $G$ to an
  ordered group is an isolated subgroup of $G$. 

  \item  Conversely, let $H$ be an isolated subgroup of $G$ and
  $g:G\to G/H$ the canonical homomorphism. Then $g(P)$ is the set of
  positive elements of an ordered group structure on $G/H$. Moreover,
  if $G$ is totally ordered, so is $G/H$. 
  \end{enumerate}
\end{prop}

If $G$ is totally ordered, then the set of isolated subgroups of $G$
are totally ordered by the order of inclusion. For otherwise, there is a
positive element $x$ in one isolated subgroup $H$ but not in $H'$, and
a positive element $x'\in H'\setminus H$. Suppose for example $x\le
x'$, then $x\in H'$, a contradiction.  

\begin{defn}\label{ht}
  \

  \begin{enumerate}
  \item 
    Let $G$ be a totally ordered group. If the number of isolated
      subgroups of $G$ distinct from $G$ is finite and is equal to
      $n$, $G$ is said to be \emph{of height $n$}. If this number is
      infinite, $G$ is said to be \emph{of infinite height}. Denote by
      $h(G)$ the height of $G$. 

  \item The \emph{height} of a valuation $v$ of $K$ is defined as the
      height of its value group.    
  \end{enumerate}

\end{defn}

The height of the groups $\Z$ and $\R$ are of height $1$. If $G$ is a
totally ordered group and $H$ is an isolated subgroup, then
$h(G)=h(H)+h(G/H)$. In particular, if $G$ is the lexicographic product
of two totally ordered groups $H$ and $H'$, then
$h(G)=h(H)+h(H')$. Thus, the lexicographic product $\Z\times \Z$ is of
height $2$. 

Fix a valuation ring of $A$ of $K$ with the canonical valuation $v_A:
K^\times\to \Gamma_A:=K^\times/A^\times$. For subring $B$ of $K$
containing $A$, $B$ is a valuation ring and $A^\times \subset
B^\times$. Let $\lambda:\Gamma_A\to \Gamma_B$ be the natural
projection. As $A\subset B$, $\lambda$ maps the positive elements of
$\Gamma_A$ to positive elements of $\Gamma_B$. Thus, $\lambda$ is a
morphism of ordered groups and the kernel $H_B$ of $\lambda$ is an
isolated subgroup of $\Gamma_A$. The mapping $B\mapsto H_B$ is an
increasing bijection of the set of subrings of of $K$ containing $A$
onto the set of isolated subgroups of $\Gamma_A$ cf. \cite[VI, \S 
  4.3, Proposition~4]{bourbaki:ca}. Combining with
  Proposition~\ref{prop:v-prime}, these two sets are in deceasing
  bijection with the set of primes ideals of $A$. In particular,
  $h(\Gamma_A)={\rm ht} (\grm(A))=\dim A$, cf. \cite[VI, \S 
  4.4, Proposition~5]{bourbaki:ca}.

A totally ordered group $G$ is of height $\le 1$ if and only if it is
isomorphic to a subgroup of $\R$, cf.~\cite[VI, \S 
4.5, Proposition~8]{bourbaki:ca}.

\subsection{Topological fields and completions}
We first discuss the notion of completeness of a Hausdorff
commutative topological group and its completion. 
Our references are \cite{bourbaki:top} and \cite[Chap.~10]{li:alg}. 

\begin{defn}
  Let $X$ be a nonempty set. 

  \begin{enumerate}
  \item 
   A \emph{filter} of $X$ is a nonempty subset
  $\grF\subset P(X)$ of the power set of $X$ satisfying
  \begin{enumerate}
  \item[F.1] If $A,B\in \grF$, then $A\cap B \in \grF$;
  \item[F.2] If $A\in \grF$ and $A\subset A'\subset X$, then $A'\in \grF$;
  \item[F.3] $\emptyset \not\in \grF$.  
  \end{enumerate}

\item  A \emph{filter base} of $X$ is a nonempty subset
  $\grB\subset P(X)$ satisfying
  \begin{enumerate}
  \item[FB.1] If $A,B\in \grF$, then there exists $C\in \grB$ such that
    $C\subset A\cap B$;
  \item[FB.2] $\emptyset \not\in \grB$.  
  \end{enumerate}
  \end{enumerate}
\end{defn}

Condition F.3 implies that any finite intersection of members in $\grF$ is
nonempty. For a filter base $\grB$, the set $\grF:=\{F\subset X:
\exists B \in \grB, B\subset F \}$ is a filter of $X$, called the
filter generated by $\grB$, and $\grB$ is called a base of $\grF$. 
A filter $\grF'$ is called a \emph{refinement} of $\grF$ if
$\grF\subset \grF'$. 

\begin{eg}
  Let $\{x_k\}_{k=1}^\infty$ be a sequence in $X$. Let $\grF$ be the
  set consisting of all subsets $E$ such that there exists $N\ge1$
  such that $x_k\in E$ for all $k\ge N$. Then $\grF$ is a filter. Let
  $\grB$ be the set consisting of all subsets $\{x_k, x_{k+1},\dots,
  \}$ for some $k$. Then $\grB$ is a filter base which generates $\grF$.
\end{eg}

\begin{eg}
Let $X$ be a topological space. For every point $x\in X$, let $\grN_x$
denote the collection of all neighborhoods $E$ of $x$ (there exists an
open subset $U\ni x$ contained in $E$). Then $\grN_x$ is a filter. Any
fundamental system of neighborhoods of $x$ is a filter base of $\grN_x$.   
\end{eg}

\begin{defn}
  Let $X$ be a topological space. 
 We say a filter base $\grF$ \emph{converges to a point $x\in X$},
    denoted by $\grF \to x$, if
    every $E\in \grN_x$ contains a member $F\in \grF$. In this case, by $FB.1$, $x$ is in the closure of every
    member $F\in \grF$. If $\grF$
    is a filter, this is equivalent to say that $\grF$ is a refinement of
    $\grN_x$. 
\end{defn}

A topological space $X$ is Hausdorff if and only if every filter
converges to at most one point. 
Let $f:X\to Y$ be a map of topological spaces and $\grF$ a filter of
$X$. Set 
\[ f\grF:=\{F\subset Y: \exists E\in \grF, f(E)\subset F \}, \]  
which is a filter as $f(A\cap B)\subset f(A)\cap f(B)$. Then the map
$f$ is continuous if and only if for every $x\in X$ one has $(\grF\to
x) \implies (f\grF\to f(x))$. 

\begin{defn}
  Let $(A,+)$ be a Hausdorff commutative topological group. 

\begin{enumerate}
  \item 
 A filter $\grF$ of $X$ is called a \emph{Cauchy filter} if for any
    neighborhood $U\in \grN_0$ there exists $E\in \grF$ such that 
\[ E-E:=\{x-y: x,y\in E\}\subset U. \]

\item $A$ is said to be \emph{complete} if every Cauchy filter converges.

\item  A \emph{completion} of $A$ is pair $(A^*, \iota)$, where $
    A^*$ is a complete topological abelian group and $\iota:A\to A^*$
    is a morphism of topological groups, satisfying the following
    conditions.
    \begin{itemize}
    \item [(a)] $\iota: A\to \iota(A)$ is a homeomorphism. 
    \item [(b)] $\iota(A) \subset  A^*$ is dense.
    \end{itemize}  

  \end{enumerate}
\end{defn}
The condition (a) says that $\iota$ is injective and topology of $A$
is the same as the topology induced by $A^*$. It is proved
\cite[III, \S  3.5, Theorem~2 and \S 3.4 Proposition 8]{bourbaki:top}
   that a completion $( A^*,\iota)$ exists and 
    satisfies the functorial property: for any pairing $(B,f)$ where
    $B$ is a complete topological abelian group and $f:A\to B$ is a
    morphism of topological groups, then there exists a unique
    morphism $g:A^* \to B$ such that $g\circ \iota =f$. In
    particular, a completion $(A^*,\iota)$ is unique up to a unique
    isomorphism; such a pair $(A^*, \iota)$ is called \emph{the
      completion of 
      $A$}. If $A$ is a Hausdorff topological ring, then the
    completion $A^*$ of $(A,+)$ is a complete topological ring
    (complete for the underlying topological group $(A^*,+)$),
    cf.~\cite[III, \S 6]{bourbaki:top}. \\

Let $v$ be a valuation of a field $K$ with value group $G$. For all
$\alpha\in G$, let
\[ V_\alpha:=\{x\in K: v(x)>\alpha \} \quad \text{and} \quad V_{\ge
  \alpha}:=\{x\in K: v(x)\ge \alpha\} \] 
which are clearly 
additive subgroups of $K$. There exists a unique linear topology
$\grT_v$ on $K$ for which the sets $V_\alpha$ form a fundamental
system of neighborhoods of $0$. 
If $v$ is trivial, then $\grT_v$ is the discrete topology.
Equipped with this topology $K$ is a Hausdorff topological field.

Let $K^*$ be the completion of $K$, which is a complete topological
ring. Note that if $v$ is of height $1$, or more generally, 
there exists a
\emph{countable} fundamental system of neighborhoods of $0$. 
Then the notion 
of completeness and the construction of completion can be made by
the usual Cauchy sequences, which is similar to the classical
construction of $\R$ from $\Q$.  


\begin{prop}\label{prop:completion}
  Let $v$ be a valuation of a field $K$ with value group $G$ and equip $G$ with the discrete topology. 

\begin{enumerate}
  \item  The complete ring $K^*$ of $K$ is a topological field.

\item  The continuous map $v:K \to G_\infty$ can be extended uniquely to a continuous map $v^*: K^* \to G_\infty$ which is a valuation of $K^*$.

\item The topology on $K^*$ is the topology defined by the valuation $v^*$.

\item For all $\alpha\in G$, the closures $\ol V_\alpha$ and
  $\ol V_{\ge \alpha}$ of   $V_\alpha$ and $V_{\ge \alpha}$ are the subsets of $K^*$ defined by $v^*(x)> \alpha$ and $v^*(x)\ge \alpha$, respectively.

\item  The valuation ring of $v^*$ is the completion $A^*$ of $A$; its maximal ideal is the completion $\grm^*$ of the maximal ideal $\grm$ of $A$.

\item  $A^*=A+\grm^*$; the residue field of $v^*$ is canonically identified with that of $v$. 

\end{enumerate}
  
\end{prop}

  \begin{proof}
    See \cite[VI, \S 
  5.3, Proposition~5]{bourbaki:ca}.
\end{proof}

\subsection{Extensions of valuations and the fundamental inequality}
\label{sec:V.3}

Let $v$ be a valuation of a field $K$, $A$ the valuation ring of $v$, 
$\grm$ its
maximal ideal, and $\Gamma_v$ the value group. 
Let $L/K$ be a finite field extension and $w$ be a
valuation of $L$ which extends $v$. 
Denote by $\Gamma_w$ the value group, $A'$ the valuation ring 
and $\grm'$ the
maximal ideal of $w$, respectively. 
Write $\kappa(v)$ and $\kappa(w)$ for the residue fields of $v$ and
$w$, respectively. The completion of $K$ at $v$ (resp.~ of $L$ at $w$)
is denoted by $K^*$ (resp.~$L_w^*$). The valuation of $K^*$ extending $v$ is denoted by $v^*$. 

\begin{defn} \ 
  \begin{enumerate}
  \item The \emph{ramification index of $w$ over $v$} is defined as
  $e(w/v):=[\Gamma_w:\Gamma_v]$.

\item  The \emph{residue class degree} of $w$ over
  $v$ is defined as $f(w/v):=[\kappa(w):\kappa(v)]$. 

  \end{enumerate}
\end{defn}

\begin{lemma}\label{lm:E.1}
  Let $K$, $v$, $L$ and $w$ be as above. 

  \begin{enumerate}
  \item 
  The inequality $e(w/v)f(w/v)\le [L:K]$ holds. 
    In particular, $e(w/v)$ and
    $f(w/v)$ are finite.

  \item  The height of $w$ is equal to that of $v$.

\item  The $w$ is trivial (resp. discrete) if and only if so is
$v$.
  \end{enumerate}
\end{lemma}
\begin{proof}
  See \cite[VI, \S 8.1]{bourbaki:ca}.
\end{proof}

\begin{defn}[{\cite[VI, \S 7.2]{bourbaki:ca}}]
  Two valuations $v$ and $v'$ of $K$ are said to be \emph{independent}
  if the subring generated by their valuation rings is equal to $K$;
  and \emph{dependent} otherwise.
\end{defn}

The trivial valuation is independent of any valuation of $K$. 
Two valuations 
are dependent if and only if 
there is a relation $A\subset A'\subsetneq K$ among
their valuation rings $A$ and $A'$. 
If two non-trivial valuations $v$ and $v'$ are of same finite height, 
then they are dependent if and only if
they are equivalent. Indeed, if $v$ and $v'$ are equivalent, then
$A=A'\subsetneq K$ and they are dependent. Conversely if they are
dependent then up to switching the order one has $A\subset A' \subsetneq K$.
Therefore, $A=A'$ for otherwise ${\rm ht} (v')<{\rm ht} (v)$, contradiction.

Let $\Sigma_v$ be a complete set of representatives of equivalence
classes of extensions of a valuation $v$ of $K$ on $L$. If $v$ is
trivial then $\Sigma_v$ consists of the trivial valuation of $L$. 
Also write $w|v$ if $w\in \Sigma_v$.
If $w_1, w_2\in \Sigma_v$ with $w_1\neq w_2$ (this implies that $v$
must be non-trivial), then there is no
inclusion relation for their valuation rings cf.~\cite[(A),
p.2]{cohen-zariski}. Thus, every two distinct valuations in $\Sigma_v$
are independent.  

\begin{prop}\label{prop:E.2}
   Let $v$ be a valuation of $K$ and $L/K$ a finite field extension.
   \begin{enumerate}
   \item 
 For every $w\in \Sigma_v$, one has $e(w^*/v^*)=e(w/v),
f(w^*/v^*)=f(w/v)$, $[L_w^*:K^*]\le [L:K]$ and $e(w/v)f(w/v)\le
[L_w^*:K^*]$.

\item 
 Every set of pairwise independent valuations of $L$ extending a
 non-trivial valuation $v$ is finite. Let $\{w_1,\dots, w_r\}$ be a
 maximal set of pairwise independent valuations of $L$ extending a
 non-trivial valuation $v$. Then the canonical mapping $\phi:
 K^*\otimes_K L \to \prod_{i=1}^r L_{w_i}^*$ (extending by continuity
 the diagonal map $L\to \prod_{i=1}^r L_{w_i}^*$) is surjective, its
 kernel is the Jacobson radical of $K^*\otimes_K L$ and  
\begin{equation}
  \label{eq:E.2}
  \sum_{i=1}^r [L_{w_i}^*: K^*]\le [L:K].
\end{equation}
   \end{enumerate}
\end{prop}
\begin{proof}
  See \cite[VI, \S 8.2, Proposition 2]{bourbaki:ca}.
\end{proof}

\begin{cor}[The fundamental inequality]\label{cor:E.3}
  Let $v$ be a valuation of $K$ and $L/K$ a finite extension. We have
  \begin{equation}
    \label{eq:E.3}
    \sum_{w|v} e(w/v)f(w/v)\le [L:K]. 
  \end{equation}
\end{cor}
\begin{proof}
  If $v$ is trivial, then $w$ is trivial and one has
  $\Gamma_v=\Gamma_w=\{0\}$, 
  $\kappa(v)=K$ and $\kappa(w)=L$. Therefore, $\sum_{w|v} e(w/v)
  f(w/v)=[L:K]$. 
  Suppose that $v$ is non-trivial. By the remark above
  Proposition~\ref{prop:E.2}, $\Sigma_v$ is a maximal set
  $\{w_1,\dots, w_r\}$ 
  of pairwise independent valuations of $L$ extending $v$.
  Then \eqref{eq:E.3} follows from Proposition~\ref{prop:E.2}. \qed   
\end{proof}

\begin{defn}[{\cite[VI, \S 3.5, \S 8.4]{bourbaki:ca}}]\label{def:ini_ind}
  \ 

  \begin{enumerate}
  \item 
 Let $G$ be an ordered set. A subset $M$ of $G$ is called \emph{major}
 if the relations $x\in M$ and $y\ge x$ imply $y\in M$.

\item  Let $G$ be a totally ordered commutative group and $H$ a
  subgroup of finite index. Denote by $G_{>0}\subset G$ the subset of
  strictly positive elements.  The \emph{initial index of $H$ in $G$},
  denoted by 
  $\varepsilon(G,H)$, is the number of major subsets $M$ of $G$ such
  that $H_{>0}\subset M \subset G_{>0}$. 

  \end{enumerate}
\end{defn}

If $G=\Z$ and $H=m\Z$ with $m>0$, letting $M(x):=\{y\in G: y\ge x\}$,
then $M(1), M(2),\dots, M(m)$ are all major subsets of $G$ satisfying
the property in Definition~\ref{def:ini_ind} and
$\varepsilon(G,H)=m$. 

\begin{prop}\label{prop:E.4}
  Let $G$ be a totally ordered commutative group and $H$ a subgroup of
  finite index. 

  \begin{enumerate}
  \item 
If the set $G_{>0}$ has no least element, then $\varepsilon(G,H)=1$;

\item   If the set $G_{>0}$ has the least element $x_0$, then
  $\varepsilon(G,H)=[G_0: (G_0\cap H)]$, where $G_0$ is the cyclic
  subgroup generated by $x_0$. 
  \end{enumerate}
  In particular, $\varepsilon(G,H)$ divides $[G:H]$.

\end{prop}
\begin{proof}
  See \cite[VI, \S 8.4, Proposition 3]{bourbaki:ca}.
\end{proof}

\begin{defn}\label{def:ini_ram}
  Let $v$ be a valuation of $K$ and $w|v$ a valuation of a finite
  extension $L/K$ with value groups $\Gamma_v$ and $\Gamma_w$,
  respectively. The \emph{initial ramification index of $w$ with
  respect to $v$ (or $w$ over $v$)} is defined as $\varepsilon(w/v):=
  \varepsilon(\Gamma_w,\Gamma_v)$.   
\end{defn}

\begin{thm}\label{thm:E.5}
  Let $L/K$ be a  and
  Let $v$ a valuation of $K$
  with valuation ring $A$ and maximal ideal $\grm$. Let $L/K$ be a
  finite field extension with integral closure $B$ of $A$ in $L$. The
  following conditions are equivalent:

  {\rm (a)} $B$ is a finite $A$-module;

  {\rm (b)} $B$ is a free $A$-module;
  
  {\rm (c)} $[B/\grm B: \kappa(\grm)]=[L:K]$;

  {\rm (d)} $\sum_{w|v} \varepsilon(w/v) f(w/v)=[L:K]$.
 
\end{thm}
\begin{proof}
  See \cite[VI, \S 8.5, Theorem 2]{bourbaki:ca}. 
\end{proof}

\begin{remark}\label{rem:E.6} \ 

  \begin{enumerate}
  \item 

 The fundamental inequality (Corollary~\ref{cor:E.3}) was first
  proved by Cohen and Zariski for an arbitrary valuation $v$
  \cite{cohen-zariski}.

\item The condition (d) is equivalent to (d') $\sum_{w|v} e(w/v)
  f(w/v)=[L:K]$ and $\varepsilon(w/v)=e(w/v)$ for all $w|v$. 
  When $v$ is discrete, one has $\varepsilon(w/v)=e(w/v)$ and 
  the condition (d) is equivalent to $\sum_{w|v} e(w/v) f(w/v)=[L:K]$. 
  In this special case, Theorem~\ref{thm:E.5} was first proved by
  Cohen and Zariski \cite{cohen-zariski}.  
  \end{enumerate}
\end{remark}


\section{Proofs of Theorems~\ref{thm:1} and~\ref{thm:2}}
\label{sec:04}

\subsection{Proof of Theorem~\ref{thm:1}}
\label{sec:41}
Theorem~\ref{thm:1} follows from Lemma~\ref{lm:fe} and
Proposition~\ref{prop:fe2}. 

 \begin{lemma}\label{lm:fe}
  Let $A$ be a Dedekind domain with quotient field $K$, and $L$ a
  finite field extension of $K$ of degree $n$ with integral closure
  $B$ of $A$ in $L$. For each nonzero prime ideal $\grp$ of $A$, one
  has 
  \begin{equation}
    \label{eq:fe2}
    \sum_{i=1}^r e_i f_i = \dim_{k(\grp)} B/\grp B \le n,
    \end{equation}    
where $\grP_1,\dots, \grP_r$ are the prime ideals of $B$ over $\grp$,
$e_i$ and $f_i$ are the ramification index and the residue class degree of
$\grP_i$ over $\grp$. Moreover, if $B$ is a finite $A$-module, then
the equality holds. 
\end{lemma}
\begin{proof}
  This is well-known (cf.~\cite{serre:lf}); we include a proof for the
  reader's convenience. 
We localize the Dedekind domains $A$ and $B$ at $\grp$ and get a
  discrete valuation ring  $A_\grp$ and a semi-local Dedekind domain
  $B_\grp$ with same number of maximal ideals $\grP_i B_\grp$. Then
  $B_\grp$ is the integral closure of $A_\grp$ in $L$ and the
  numerical invariants $r, e_i, f_i$ remain the same.  Therefore,
  after replacing $A$ by $A_\grp$,  we can assume that $A$ is a
  discrete valuation ring with uniformizer $\pi$. 
  By the Chinese Remainder Theorem, $B/\grp B\simeq \prod_{i=1}^r
  B/\grP_i^{e_i}$. 
  We filter each $k(\grp)$-vector space $B/\grP_i^{e_i}$ by the
  decreasing subspaces $\grP_i^j/\grP_i^{e_i}$ and obtain 
  \begin{equation}
    \label{eq:4.1}
    \dim_{k(\grp)} B/\grp B = \sum_{i=1}^r \sum_{j=0}^{e_i-1}
    \dim_{k(\grp)} \grP_i^{j}/\grP_i^{j+1}.     
  \end{equation}
  The ideal $\grP_i^j/\grP_i^{j+1}$ generated by one element $a_j$ as
  the quotient ring $B/J$ for any nonzero ideal $J$ is a principal
  ideal ring. 
  The map
  $1 \mapsto a_j$ induces an isomorphism $A/\grP_i\simeq
  \grP_i^j/\grP_i^{j+1}$ and hence $\dim_{k(\grp)}
  \grP_i^j/\grP_i^{j+1}=f_i$. It follows from \eqref{eq:4.1} that
  $\dim_{k(\grp)} B/\grp B = \sum_{i=1}^r e_i f_i$. 

 We prove the inequality in \eqref{eq:4.1} by showing that if $\bar
 x_1, \dots \bar x_s$ are $k(\grp)$-linearly independent, then their
 liftings $x_1,\dots x_s$ are $K$-linearly independent. Suppose not,
 then 
  \[ a_1 x_1 + a_2 x_2 +\dots +a_s x_s =0 \]
  for some nonzero element $a_i$ in $K$. Multiplying a suitable power
  of $\pi$, we can assume that $a_i\in A$ for all $i$ but $a_i \not
  \in \grp$ for some $i$. Modulo $\grp$ 
we get a non-trivial linear relation $\bar a_1 \bar x_1+\bar a_2 \bar x_2+\dots +\bar a_s
  \bar x_s=0$, a contradiction. 

Suppose $B$ is a finite $A$-module. Since $B$ is torsion free and $A$
is a principal ideal domain, 
$B$ is a free $A$-module of rank $s$. Then one has $\dim_K L= \rank_A
B=\dim_{k(\grp)} B/\grp B$. This proves the desired equality. \qed     
\end{proof}   

\begin{prop}\label{prop:fe2}
  Let the notation be as in Lemma~\ref{lm:fe}. Then
  \begin{equation}
    \label{eq:fe3}
    \dim_{k(\grp)} B/\grp B =\sum_{i=1}^r \rank_{A_\grp^*}
    B_{\grP_i}^* =  \sum_{i=1}^r \dim_{K_\grp^*} L_{\grP_i}^* =
    [(L\otimes_K K_\grp^*)_{\rm ss}: K_\grp^*],   
  \end{equation}
  where $K^*_\grp$ and $A_\grp^*$ (resp.~ $L_{\grP_i}^*$ and
  $B_{\grP_i}^*$) denote the completion of $K$ and $A$ (resp.~$L$ and
  $B$) at $\grp$ (resp. $\grP_i$), respectively.  
\end{prop}
\begin{proof}
  The ramification index and residue class degree remain the same after the
  completion, that is, $e(\grP_i/\grp)=e(\grP_i^*/\grp^*)$ and
  $f(\grP_i/\grp)=f(\grP_i^*/\grp^*)$, where $\grp^*=\grp A_\grp^*$
  and $\grP_i^*=\grP_i B_{\grP_i}^*$. Since $K_{\grp}^*$ is complete
  and $L_{\grP_i}^*$ is a finite field extension of $K_\grp^*$, we
  have $e_if_i=e(\grP_i^*/\grp^*)f(\grP_i^*/\grp^*)=[
  L_{\grP_i}^*:K_\grp^*]$ \cite[Chap.~II, \S 2, Corollary
  1]{serre:lf}. Note that $B_{\grP_i}^*$ is the integral closure of
  $A_\grp^*$ in $L^*_{\grP_i}$. Since $B^*_{\grP_i}$ is a finite free
  $A_\grp^*$-module say of rank $m$, we have
  $L_{\grP_i}^*=B_{\grP_i}^* \otimes_{A_\grp^*} K_\grp^*\simeq
  (K_\grp^*)^m$ 
  and $\rank_{A_\grp^*} B^*_{\grP_i}=[L_{\grP_i}^*:K_\grp^*]=e_i
  f_i$. This and Lemma~\ref{lm:fe} prove the first two equalities.
   The last equality follows from the isomorphism
  $\prod_{i=1}^r L_{\grP_i}^* \simeq (L\otimes_K K_\grp^*)_{\rm ss}$; 
  see 
  Proposition~\ref{prop:E.2}.  \qed
\end{proof}

\subsection{Proof of Theorem~\ref{thm:2}}
\label{sec:42}

Theorem~\ref{thm:2} will follow immediately from Theorem~\ref{thm:JE}.


\begin{prop}\label{prop:fe_g}
  Let $A$ be a Dedekind domain with quotient field $K$. Then the
  following statements are equivalent: 
  \begin{enumerate}
  \item For any finite field extension $L/K$ and any nonzero prime
    ideal $\grp$ of $A$, one has $\sum_{\grP|\grp} e_\grP
    f_\grP=[L:K]$. 
  \item For any finite field extension $L/K$ and any nonzero prime
    ideal $\grp$ of $A$, the tensor product $L\otimes_K K_\grp^*$ is a
    semi-simple $K_\grp^*$-algebra. 
  \item $A$ is a $G$-ring.
  \item For every nonzero prime ideal $\grp$ of $A$, 
  the localization $A_\grp$ is Japanese.
  \end{enumerate}
\end{prop}

\begin{proof}
The equivalence of (1) and (2) follows from Theorem~\ref{thm:1}.

We prove the equivalence of (2) and (3). 
To show $A$ is a $G$-ring, one needs to show that every formal fiber
of $\phi_\grp: A_\grp \to A_\grp^*$ is geometrically regular. If
$\grp=0$, the formal fiber $K\to K^*=K$ is clearly geometrically
regular. Suppose $\grp$ is a nonzero prime ideal. The special
formal fiber   
$k(\grp) \to k(\grp^*)=k(\grp)$ is clearly geometrically regular and
the generic formal fiber is given by $K\to K_\grp^*$. By
Lemma~\ref{lm:geom-reg}, $K_\grp^*/K$ is geometrically regular if and
only if it is separable. 
Thus, $A$ is a $G$-ring if and only if for any nonzero prime ideal
$\grp$ of $A$, the field extension $K_\grp^*/K$ is separable. This is
equivalent to that for any finite extension $L/K$ and for any nonzero
prime ideal $\grp$ of $A$, the tensor product $L\otimes_K K^*_\grp$ is
semi-simple. 

We prove the equivalence of (1) and (4). The direction
(4)$\implies$(1) follows from Lemma~\ref{lm:fe} and we show the other
direction. For each nonzero prime ideal $\grp$, since $\sum_{i=1}^r 
e_i f_i=[L:K]$, by~Theorem~\ref{thm:E.5}
the integral closure $B_\grp$ of $A_\grp$ in $L$ is
a finite $A_\grp$-module.  
Thus, for any finite field extension
$L/K$ the integral closure of $A_\grp$ in $L$ is a finite $A$-module,
and hence $A_\grp$ is Japanese. This completes the proof of the
proposition. \qed



\end{proof}



\begin{thm}\label{thm:JE}
  Any Japanese Dedekind domain $A$ is excellent.
\end{thm}
\begin{proof}
  It follows immediately from Proposition~\ref{prop:fe_g} and
   Lemma~\ref{lm:japan-j2} that $A$ is quasi-excellent. It is
  well-known that any Dedekind domain is universally
  catenary (cf.~Remark~\ref{rem:catenary}). 
  Therefore, $A$ is excellent.  \qed
\end{proof}

\section*{Acknowledgements}
The author thanks Raymond Heitmann who kindly informed him an example of 
a locally Nagata principal ideal domain which is not
Nagata~\cite{heitmann:loc_nagata_not_nagata}. He also thanks the
referee for a careful reading and the reference~\cite{morel:adic}.
The author is partially support by the MoST grant 109-2115-M-001-002-MY3.

\end{document}